\documentclass[12pt,a4paper]{article}
\usepackage{amsthm,amsmath,amssymb}
\usepackage{nicefrac}
\usepackage{enumerate}
\usepackage{listings}
\usepackage{hyperref}
\allowdisplaybreaks
\newtheorem{thm}{Theorem}
\newtheorem{lemma}[thm]{Lemma}

\theoremstyle{remark}

\theoremstyle{definition}

\newcommand{\incfree}{\bar{\alpha}}
\newcommand{\isop}{i_V}

\newcommand{\vkl}[1]{$(v,k,\lambda)$}

\usepackage{color}

\definecolor{VeryLightBlue}{rgb}{0.9,0.9,1}
\definecolor{LightBlue}{rgb}{0.8,0.8,1}
\definecolor{MidBlue}{rgb}{0.5,0.5,1}
\definecolor{DarkBlue}{rgb}{0,0,0.6}
\definecolor{Blue}{rgb}{0,0,1}

\definecolor{Gold}{rgb}{1,0.843,0}

\definecolor{LightGreen}{rgb}{0.88,1,0.88}
\definecolor{MidGreen}{rgb}{0.6,1,0.6}
\definecolor{DarkGreen}{rgb}{0,0.6,0}

\definecolor{VeryLightYellow}{rgb}{1,1,0.9}
\definecolor{LightYellow}{rgb}{1,1,0.6}
\definecolor{MidYellow}{rgb}{1,1,0.5}
\definecolor{DarkYellow}{rgb}{1,1,0.2}

\definecolor{DarkPurple}{rgb}{.6,0,1}

\definecolor{Red}{rgb}{1,0,0}
\definecolor{VeryLightRed}{rgb}{1,0.9,0.9}
\definecolor{LightRed}{rgb}{1,0.8,0.8}
\definecolor{MidRed}{rgb}{1,0.55,0.55}

\newcommand{\be}{\begin{equation}}
\newcommand{\ee}{\end{equation}}
\newcommand{\bea}{\begin{eqnarray}}
\newcommand{\eea}{\end{eqnarray}}
\newcommand{\bean}{\begin{eqnarray*}}
\newcommand{\eean}{\end{eqnarray*}}

\def\PG{\mathrm{PG}}

\def\qed{\hfill$\Box$\vspace{12pt}}

\usepackage[normalem]{ulem}

\title{The isoperimetric number of the incidence graph of $\PG(n,q)$}
\author{Andrew Elvey-Price\thanks{E-mail: a.elveyprice@student.unimelb.edu.au},\ Muhammad Adib Surani\thanks{E-mail: m.surani@student.unimelb.edu.au}\; and Sanming Zhou\thanks{E-mail: sanming@unimelb.edu.au}\\ \\
	School of Mathematics and Statistics\\
	The University of Melbourne\\
	Parkville, VIC 3010, Australia}

\begin{document}
	\openup 0.5\jot 
	\maketitle
	
	\begin{abstract}
Let $\Gamma_{n,q}$ be the point-hyperplane incidence graph of the projective space $\PG(n,q)$, where $n \ge 2$ is an integer and $q$ a prime power. We determine the order of magnitude of $1-\isop(\Gamma_{n,q})$, where $\isop(\Gamma_{n,q})$ is the vertex-isoperimetric number of $\Gamma_{n,q}$. We also obtain the exact values of $\isop(\Gamma_{2,q})$ and the related incidence-free number of $\Gamma_{2,q}$ for $q \le 16$.

\medskip

{\it Keywords:}~ isoperimetric number; vertex-isoperimetric number; incidence-free number; projective plane; projective space  

{\it AMS subject classification (2010):}~ 05C40, 05B25
	\end{abstract}

\section{Introduction}
\label{sec:int}

A fundamental problem in graph theory is to understand various expansion properties of graphs. The expansion of a graph is commonly measured by its isoperimetric number, also known as the Cheeger constant, or its vertex-isoperimetric number. These two parameters have been studied extensively, especially in the study of expanders, and a number of results on them exist in the literature (see for example \cite{HLW2006}). A major concern is to produce good (sharp) lower bounds for these isoperimetric numbers and related invariants. Such isoperimetric inequalities are closely related to problems in probabilistic combinatorics, theoretical computer science, spectral graph theory, etc.  

Inspired by a conjecture of Babai and Szegedy, in this paper we study the vertex-isoperimetric number of the point-hyperplane incidence graph of the projective space $\PG(n, q)$.

Let $\Gamma=(V,E)$ be a graph. The \emph{vertex-boundary} $N(X)$ of a subset $X \subseteq V$ is the set of vertices in $V \setminus X$ that are adjacent to at least one vertex in $X$. The \emph{vertex-isoperimetric number} of $\Gamma$ is defined \cite{HLW2006} as
	\be
	\label{eq:iv}
	\isop(\Gamma) = \min \left\{\frac{|N(X)|}{|X|} : \emptyset \neq X \subseteq V, |X| \leq \frac{|V|}{2} \right\}.
	\ee
The problem of determining the vertex-isoperimetric number of a graph is known to be NP-complete. There are very few families of graphs whose vertex-isoperimetric numbers have been computed exactly (see e.g. \cite{Harper99}). The reader is referred to \cite{Harper2004} for the history of this problem and related results. The related problem of determining $\min \{|N(X)|: \emptyset \neq X \subseteq V\}$ has also been studied extensively (\cite{Harper2004, Leader91}); see, for example, \cite{Harper66} for Harper's classical result on this problem for hypercubes and \cite{BL90} for an isoperimetric inequality for the discrete torus. 

As pointed out in \cite{harper1994isoperimetric}, many isoperimetric problems can be put into the form for bipartite graphs. In this case a closely related parameter is as follows. 
	Let $\Gamma$ be a bipartite graph with bipartition $\{V_1, V_2\}$ such that $|V_1| = |V_2|$. If $S \subseteq V_1$ and $T \subseteq V_2$ are such that $|S|=|T|$ and there is no edge of $\Gamma$ between $S$ and $T$, then $(S, T)$ is called an \emph{incidence-free pair}. The \emph{incidence-free number} of $\Gamma$, first introduced in \cite{de2012large} and denoted by $\incfree(\Gamma)$, is defined to be the maximum size of $S$ among all incidence-free pairs $(S, T)$. That is,  
	$$\incfree(\Gamma) = \max_{S \subseteq V_1} \min\{ |S|, |V_2 \setminus N(S)| \}.$$
	This parameter is particularly useful for bounding $\isop(\Gamma)$ for bipartite graphs $\Gamma$. Indeed, for any incidence-free pair $(S, T)$ with $|S| = \incfree(\Gamma)$, by setting $X = S \cup (V_2 \setminus T)$ in \eqref{eq:iv} we obtain
	\begin{equation}\label{eq:incfreeisop}
	i_V(\Gamma) \leq 1 - \frac{\incfree(\Gamma)}{|V_1|}.
	\end{equation}

 		The \emph{incidence graph} (or Levi graph) \cite{Beth-Jung-Lenz} of a $2$-$(v, k, \lambda)$ design $D$ is the bipartite graph with one part of the bipartition consisting of the points of $D$ and the other part the blocks of $D$ such that a point is adjacent to a block if and only if they are incident in $D$. Obviously, if $D$ is a symmetric design, then its incidence graph is a $k$-regular bipartite graph with $v$ vertices in each part such that any two vertices in the same part have exactly $\lambda$ common neighbours in the other part. Conversely, any $k$-regular bipartite graph with these properties is isomorphic to the incidence graph of a symmetric $2$-$(v, k, \lambda)$ design. Such a graph $\Gamma$ is called a \emph{$(v,k,\lambda)$-graph} and its bipartition is denoted by $\{V_1(\Gamma), V_2(\Gamma)\}$. We require $k$ to be a positive integer but we allow the degenerate case $\lambda=0$ for which the graph is a perfect matching. It is well known \cite{Beth-Jung-Lenz} that the parameters $(v,k,\lambda)$ for a symmetric design satisfy 
	\begin{equation}
	\label{eq:vkl}
	\lambda(v-1) = k(k-1).
	\end{equation}

Throughout the paper we use $\Gamma_{n,q}$ to denote the incidence graph of the point-hyperplane design of the projective space $\PG(n,q)$, where $n$ is a positive integer and $q$ a prime power. More explicitly, let $V_1$ and $V_2$ be the sets of 1-dimensional and $n$-dimensional subspaces of $\mathbb{F}_q^{n+1}$ respectively, where $\mathbb{F}_q$ is the finite field of order $q$. $\Gamma_{n,q}$ is the bipartite graph with bipartition $\{V_1, V_2\}$ and adjacency relation giving by subspace containment. Alternatively, we can write 
	\begin{gather*}
	V_1 = \{ \left<u\right> : u \in \mathbb{F}_q^{n+1}\},\;\,
	V_2 = \{ \left<v\right>^\perp : v \in \mathbb{F}_q^{n+1}\}\\
	\left<u\right> \mbox{ and } \left<v\right>^\perp \mbox{ are adjacent in } \Gamma_{n,q} \mbox{ if and only if } u \cdot v = 0.
	\end{gather*}
Since $\PG(n,q)$ is a symmetric $2$-$\left(\frac{q^{n+1}-1}{q-1}, \frac{q^{n}-1}{q-1}, \frac{q^{n-1}-1}{q-1}\right)$ design \cite{Beth-Jung-Lenz}, it follows that $\Gamma_{n,q}$ is a $\left(\frac{q^{n+1}-1}{q-1}, \frac{q^{n}-1}{q-1}, \frac{q^{n-1}-1}{q-1}\right)$-graph. It is readily seen that $\Gamma_{n,q}$ has diameter $2$.
	
	

Considerable interest in $\Gamma_{n,q}$ arises from algebraic graph theory and finite geometry. For example, it is known \cite{DMM08} that these graphs form a major subfamily of the family of $2$-arc transitive Cayley graphs of dihedral groups. (A graph is $2$-arc transitive if its automorphism group is transitive on the set of oriented paths of length $2$.) In \cite{BS92}, Babai and Szegedy conjectured that there is a positive absolute constant $c$ such that any finite $2$-arc transitive graph with diameter $d$ has vertex-isoperimetric number at least $c/\sqrt{d}$. They wrote further that ``it would be interesting to find reasonable symmetry conditions which would imply an expansion rate of $\Omega(1/\sqrt{d})$". The main result in the present paper (Theorem \ref{thm:order} below) is in line with this conjecture and provides a new family of symmetric graphs with expansion rate at least $\Omega(1/\sqrt{d})$.

As noted in \cite{ure1996study}, $\isop(\Gamma_{2,q})$ is closely related to arcs in the projective plane $\PG(2,q)$. Given integers $k, d > 1$, a \emph{$(k; d)$-arc} in $\PG(2,q)$ is a set of $k$ points, of which no $d+1$ are collinear. It is known that $k \le (d-1)(q+1) + 1$ for any $(k; d)$-arc in $\PG(2,q)$; a $(k; d)$-arc is \emph{maximal} if equality holds. A $(k; 2)$-arc is usually called a \emph{$k$-arc}.

Several results on $i_V(\Gamma_{n,q})$ and related problems exist in the literature. Harper and Hergert \cite{harper1994isoperimetric} and Ure \cite{ure1996study} studied the related problem of finding the minimum $|N(X)|$ for a subset $X$ of points with a given size in the projective plane $\PG(2,q)$. In  \cite{lanphier2006expansion}, Lanphier et al. studied the isoperimetric number of $\Gamma_{n,q}$. In \cite{mubayi2007independence}, Mubayi and Williford studied the independence number of the quotient of $\Gamma_{n,q}$ with respect to the partition each of whose part consists of a point of $\PG(n, q)$ and its dual hyperplane. De Winter et al. \cite{de2012large} and Stinson \cite{stinson2013nonincident} studied the incidence-free number of $\Gamma_{n,q}$. 

Determining the precise value of $\isop(\Gamma_{n,q})$ turns out to be a very challenging problem, even in the case when $n=2$ and $q$ is small. In this paper we will first prove the following bounds for $i_V(\Gamma_{n,q})$ and thus determine the order of magnitude of $1-\isop(\Gamma_{n,q})$. We will then determine the exact values of $i_V(\Gamma_{2,q})$ for all prime powers $q \le 16$.
		
	\begin{thm}
		\label{thm:order}
		Let $n\geq 2$ be an integer, $q=p^e$ a prime power and $\epsilon > 0$ a real number with $0 < \epsilon < \frac{1}{4}$. Then
		\begin{equation}
		\label{eq:order}
		\isop(\Gamma_{n,q})=1-c_{n,q}\frac{q^\frac{n+1}{2}(q-1)}{q^{n+1}-1}
		\end{equation}
		for some real number $c_{n,q}$ with $\frac{1}{2} - O(p^{\epsilon-\frac{1}{4}}) \le c_{n,q} < 1$.  
	\end{thm}
	
The upper bound $c_{n,q} < 1$ is best possible. Indeed, due to the existence of Denniston maximal arcs \cite{denniston1969some} in $\PG(2,2^{2k})$, we necessarily have that $\incfree(\Gamma_{2,2^{2k}}) = 2^{3k}-2^{2k}+2^k$, so $c_{2,2^{2k}}$ can be forced arbitrarily close to 1 for sufficiently large $k$. We suspect that the lower bound for $c_{n,q}$ can be improved to $\frac{1}{2}$ without the need for an error term.

Theorem \ref{thm:order} will be proved in the next two sections: In section \ref{sec:lb} we give a lower bound for $\isop(\Gamma)$ for any $(v, k, \lambda)$-graph $\Gamma$ and use it to prove the lower bound for $\isop(\Gamma_{n,q})$ as given in \eqref{eq:order}. In section 3 we obtain a lower bound for $\incfree(\Gamma_{n,q})$ and thus the required upper bound in \eqref{eq:order} by using \eqref{eq:incfreeisop}.

As far as we know, no exact value of $\isop(\Gamma_{2,q})$ is known even for small $q$, and the exact value of $\incfree(\Gamma_{2,q})$ is known \cite{ure1996study} only for $q \in \{2, 3, 4, 5, 7\}$. The next result gives the exact values of $\incfree(\Gamma_{2,q})$ for $q \in \{8, 9, 11, 13, 16\}$ and $\isop(\Gamma_{2,q})$ for all prime powers $q \le 16$. We will prove this result in section \ref{sec:small}. 

	\begin{thm}
		\label{thm:16}
		Let $q \le 16$ be a prime power. Then the values of $i_V(\Gamma_{2,q})$ and $\incfree(\Gamma_{2,q})$ are as given in Table \ref{tab:small}. Moreover, the equality in (\ref{eq:incfreeisop}) holds for $\Gamma_{2,q}$. That is,
		\begin{equation}
		\label{eq:16}	
		i_V(\Gamma_{2,q}) = 1 - \frac{\incfree(\Gamma_{2,q})}{q^2+q+1}.
		\end{equation} 
	\end{thm}
	
	
	\begin{table}
		\begin{center}
			\caption{The cases when $q \le 16$}
			\bigskip
			\begin{tabular}{|c|c|c|c|}
				\hline $q$ & $\incfree(\Gamma_{2,q})$ & $\isop(\Gamma_{2,q})$ & $c_{2,q}$ \text{ (3 decimal places)} \\ 
				\hline 2 & 2 & \textbf{\nicefrac{5}{7}} & 0.707 \\ 
				\hline 3 & 3 & \textbf{\nicefrac{10}{13}} & 0.577 \\ 
				\hline 4 & 6 & \textbf{\nicefrac{5}{7}} & 0.750 \\ 
				\hline 5 & 7 & \textbf{\nicefrac{24}{31}} & 0.626 \\ 
				\hline 7 & 13 & \textbf{\nicefrac{44}{57}} & 0.702 \\ 
				\hline 8 & \textbf{16} & \textbf{\nicefrac{57}{73}} & 0.707 \\ 
				\hline 9 & \textbf{19} & \textbf{\nicefrac{72}{91}} & 0.703 \\ 
				\hline 11 & \textbf{28} & \textbf{\nicefrac{15}{19}} & 0.767 \\ 
				\hline 13 & \textbf{36} & \textbf{\nicefrac{49}{61}} & 0.768 \\ 
				\hline 16 & 52 & \textbf{\nicefrac{17}{21}} & 0.813 \\ 
				\hline 
			\end{tabular}
			\label{tab:small}
		\end{center}
	\end{table}

In Table \ref{tab:small} new results from this paper are highlighted in bold. We include $c_{2,q}$ in order to estimate how close it is to the given bounds $0.5 \lesssim c_{2,q} < 1$.

	\section{Lower bounds for $i_V(\Gamma_{n,q})$}
	\label{sec:lb}

	Let $\Gamma$ be a $(v,k,\lambda)$-graph. It is known \cite{mullin1975regular} that $|N(S)| \geq \frac{k^2 |S|}{k+\lambda(|S|-1)}$ for any $S \subseteq V_1(\Gamma)$ or $S \subseteq V_2(\Gamma)$. 
	We will prove a stronger bound in the following lemma. This was already known in \cite{ure1996study} for projective planes, and here we generalise it to all symmetric 2-designs.
	\begin{lemma}
		\label{lem:onesidedbound}
		Let $\Gamma$ be a \vkl{}-graph and let $m = \left\lfloor \frac{\lambda(|S|-1)}{k} \right\rfloor + 1$. Then for any non-empty subset $S$ of $V_1(\Gamma)$ or $V_2(\Gamma)$ we have 
		\be
		\label{eq:ure}
		\frac{|N(S)|}{|S|} \geq \frac{2km-\lambda(|S|-1)}{m (m + 1)}
		\geq \frac{k^2}{k+\lambda(|S|-1)}.
		\ee
	\end{lemma}
	\begin{proof}
		Since the roles of $V_1(\Gamma)$ and $V_2(\Gamma)$ are symmetric, we may assume $S \subseteq V_1(\Gamma)$ without loss of generality. Let $T_i$ denote the number of vertices in $V_2(\Gamma)$ adjacent to exactly $i$ vertices in $S$. By counting the number of edges from $S$ to $N(S)$, as well as the number of paths of length $2$ from $S$ to $S$, in two different ways, we obtain $\sum_{i=1}^k T_i = |N(S)|$, $\sum_{i=1}^k i T_i = k|S|$ and $\sum_{i=1}^k i(i-1)T_i = \lambda |S|(|S|-1)$. 
		Since $n^2 \ge n$ for all integers $n$, we have
		\begin{align*}
		0 &\le \sum_{i=1}^k \left( (i-m)^2 - (i-m) \right) T_i\\
		&= \sum_{i=1}^k \left( i(i-1) - 2mi + m(m+1) \right) T_i\\
		&= \lambda |S| (|S|-1) - 2mk|S| + m(m+1)|N(S)|.
		\end{align*}
		Hence
		$$\frac{|N(S)|}{|S|} \geq \frac{2km-\lambda(|S|-1)}{m (m + 1)}.$$
		Since $\frac{\lambda(|S|-1)}{k} \le m \le \frac{\lambda(|S|-1)}{k}+1$ by the definition of $m$, we have
		$$
		\left(\frac{\lambda(|S|-1)}{k}+1-m\right)\left(m-\frac{\lambda(|S|-1)}{k}\right) \ge 0,
		$$ 
		that is,
		$$
		\left(\frac{\lambda(|S|-1)}{k}+1\right)\left(2m-\frac{\lambda(|S|-1)}{k}\right) \ge m(m+1).
		$$
		This yields
		$$
		\frac{2km-\lambda(|S|-1)}{m (m + 1)} \geq \frac{k^2}{k+\lambda(|S|-1)}
		$$
		as required.
		\qed
	\end{proof}
	
	Define $f: [0,v]\rightarrow[0,2v]$ by
	\begin{equation}
	\label{eq:f}
	f(x) = \frac{k^2 x}{k+\lambda(x-1)}+x.
	\end{equation}
 
	\begin{thm}
		\label{thm:f}
		Let $\Gamma$ be a \vkl{}-graph.
		Then for any $\emptyset \neq X \subseteq V(\Gamma)$ we have
		$$|N(X)| \geq f(|X|-f^{-1}(|X|)) - |X|.$$
	\end{thm}
	
	\begin{proof}
		By taking the first and second derivatives, one can see that $f$ is increasing, concave and bijective. So $f^{-1}$ exists and is increasing and convex.
		
		Denote $S=X\cap V_1(\Gamma)$ and $T=X\cap V_2(\Gamma)$. We may assume $|S| \le |T|$ without loss of generality. By Lemma \ref{lem:onesidedbound}, 
		\begin{align*}
		|N(X)| &= |N(S) \setminus T| + |N(T) \setminus S|\\
		&\geq \max\{|N(S)| - |T|, 0\} + \max\{|N(T)| - |S|, 0\}\\
		&\geq \max\{f(|S|) - |S| - |T|, 0\} + \max\{f(|T|) - |T| - |S|, 0\}.
		\end{align*}
		
		If $|S| \leq f^{-1}(|X|)$, then $|N(X)| \geq f(|X| - |S|) - |X|
		\geq f\left(|X| - f^{-1}(|X|)\right) - |X|$, as claimed.
		
		
		Otherwise assume $|S| > f^{-1}(|X|)$. Then $|X| - f^{-1}(|X|) > |X|-|S| = |T| \ge |S|$. 
		By the concavity of $f$, we have
		\begin{align*}
		f(|S|) + f(|X|-|S|) &\ge f\left(f^{-1}(|X|)\right) + f\left(|X|-f^{-1}(|X|)\right)\\
		&= |X| + f\left(|X|-f^{-1}(|X|)\right).
		\end{align*}
		It follows that
		\begin{align*}
		|N(X)| &\geq f(|S|)-|X| + f(|X| - |S|)-|X|\\
		&\geq f\left(|X|-f^{-1}(|X|)\right) - |X|.
		\end{align*}
		\qed
	\end{proof}

We now use Theorem \ref{thm:f} to prove the following lower bound for any \vkl{}-graph. 

	\begin{thm}\label{thm:kmubound}
		Let $\Gamma$ be a \vkl{}-graph and $\mu = \sqrt{k - \lambda}$. Then $$\isop(\Gamma) \geq (k-\mu)\frac{k^2+\mu^2}{k^3+\mu^3}.$$
	\end{thm}
	
	\begin{proof}
		Let $f$ be as defined in (\ref{eq:f}). Using \eqref{eq:vkl} it can be verified that $f^{-1}(v) = \frac{\mu v}{k+\mu}$. Note that, for any $x \in [0,v]$ and $\alpha \in [0,1]$, $f(\alpha x) = f(\alpha x + (1-\alpha)0) \ge \alpha f(x) + (1-\alpha)f(0) = \alpha f(x)$. Similarly, $f^{-1}(\alpha x) \le \alpha f^{-1}(x)$. Thus, for any $\emptyset \neq X \subseteq V(\Gamma)$ with $|X| \leq \frac{|V(\Gamma)|}{2} = v$, by Theorem \ref{thm:f}, 
		\begin{align*}
		\frac{|N(X)|}{|X|} &\geq -1 + \frac{1}{|X|}f\left(|X|-f^{-1}(|X|)\right)\\
		&\geq -1 + \frac{1}{|X|}f\left(\frac{|X|}{v}\left(v-f^{-1}(v)\right)\right)\\
		&\geq -1 + \frac{1}{v}f\left(v-f^{-1}(v)\right)\\
		&= -1 + \frac{1}{v}f\left(v - \frac{\mu v}{k+\mu}\right)\\
		&= (k-\mu)\frac{k^2+\mu^2}{k^3+\mu^3},
		\end{align*}
		where the last equality is obtained by a straightforward evaluation of $f$ at $\frac{kv}{k+\mu}$ by using (\ref{eq:vkl}).
		\qed
	\end{proof}

Equipped with Theorem \ref{thm:kmubound}, we are now ready to prove the lower bound for $\isop(\Gamma_{n,q})$ as stated in Theorem \ref{thm:order}.

\medskip	
	\begin{proof}\textbf{of Theorem \ref{thm:order} (lower bound)}~		 
		Let $n \geq 2$ be an integer and $q$ a prime power. It suffices to prove
		$$
		\isop(\Gamma_{n,q}) > 1-\frac{q^\frac{n+1}{2}(q-1)}{q^{n+1}-1}.
		$$	
		Let $(v,k,\lambda) = \left(\frac{q^{n+1}-1}{q-1},\frac{q^n-1}{q-1},\frac{q^{n-1}-1}{q-1}\right)$ be the parameters of $\Gamma_{n,q}$, and let $\mu = \sqrt{k-\lambda} = q^\frac{n-1}{2}$.  
		
		\medskip
		\textsf{Case 1:} $(n,q)=(2,2)$. Let $\emptyset \ne X \subset V(\Gamma_{2,2})$ be such that $|X| \le v  = 7$ and $\isop(\Gamma_{2,2}) = \frac{|N(X)|}{|X|}$. Let $f$ be as defined in (\ref{eq:f}). If $|X| \le 6$, then by the same reasoning as in Theorem \ref{thm:kmubound} we obtain
		\begin{align*}
		\frac{|N(X)|}{|X|} &\ge -1 + \frac{1}{|X|}f\left(|X|-f^{-1}(|X|)\right)\\
		&\ge -1 + \frac{1}{6}f\left(6-f^{-1}(6)\right)\\
		&= \frac{11}{\sqrt{552}}\left(79-5\sqrt{73}\right)\\
		&>1-\frac{2\sqrt{2}}{7}.
		\end{align*}
If $|X|=7$, then $|N(X)| \ge \left\lceil f\left( 7-f^{-1}(7))-7 \right)\right\rceil = 5$ and hence
$$
\frac{|N(X)|}{|X|} \ge \frac{5}{7} > 1 - \frac{2\sqrt{2}}{7}.
$$
		
		\textsf{Case 2:} $(n,q)=(2,3)$. Plugging $k=4$ and $\mu=\sqrt{3}$ into Theorem \ref{thm:kmubound}, we obtain
		$$\isop(\Gamma_{2,3}) \ge \left(4-\sqrt{3}\right)\frac{16+3}{64+3\sqrt{3}}>1-\frac{3\sqrt{3}}{13}.$$
		
		\textsf{Case 3:} $n \geq 3$ or $q \geq 4$. In this case we have
		$k-\lambda=\mu^2 = q^{n-1} \geq 4$ and so $k^2 \geq k(4+\lambda) > 4k-4\lambda = 4\mu^2$. Thus $k-\mu > \mu \geq 2$. Therefore,
		$$\frac{k^2-k \mu+2\mu^2}{\mu(k-\mu)^2} = \frac{1}{\mu} + \frac{1}{k-\mu} + \frac{2\mu}{(k-\mu)^2} < 2 \leq q.$$
		This together with Theorem \ref{thm:kmubound} implies
		\begin{align*}
		\isop(\Gamma_{n,q}) &\geq (k-\mu)\frac{k^2+\mu^2}{k^3+\mu^3}\\
		&= 1 - \frac{\mu}{k + \frac{\mu(k-\mu)^2}{k^2-k \mu+2\mu^2}}\\
		&> 1 - \frac{\mu}{k + \frac{1}{q}}\\
		&= 1-\frac{q^\frac{n+1}{2}(q-1)}{q^{n+1}-1}.
		\end{align*}
		\qed
	\end{proof}

	\section{Upper bounds for $i_V(\Gamma_{n,q})$}
	\label{sec:ub}
	 	
The following results are taken from \cite[Corollary 10]{de2012large}, \cite[Corollary 14]{de2012large} and \cite[Theorem 5]{mubayi2007independence}. We will use them in the proof of the upper bound for $i_V(\Gamma_{n,q})$ as stated in Theorem \ref{thm:order}.

	\begin{lemma}
		\label{lem:pq}
		Let $p$ be a prime and $q$ a prime power. Then 
		\begin{enumerate}[\rm (a)]
			\item $\incfree(\Gamma_{1,q}) \geq \frac{q}{2}$;
			\item $\incfree(\Gamma_{2,p}) \geq \frac{120}{73\sqrt{73}}p^\frac{3}{2}$;
			\item $\incfree(\Gamma_{2,p^{2k}}) \geq \frac{p^{3k}}{2}$ for all positive integers $k$;
			\item $\incfree(\Gamma_{2,p^{2k+1}}) \geq p^{3k}\incfree(\Gamma_{2,p})$ for all positive integers $k$; and
			\item $\incfree(\Gamma_{n+2,q}) \geq q \incfree(\Gamma_{n,q})$ for all positive integers $n$.
		\end{enumerate}
	\end{lemma}
 
To establish the upper bound in Theorem \ref{thm:order} we will also use some known results on the well-known circle problem and its primitive version. For any real number $r > 0$, define
	$$
	C(r) = \left\{(x, y) \in \mathbb{Z}^2 : x^2 + y^2 \le r  \right\} 
	$$
	and
	$$
	C'(r) = \left\{(x, y) \in \mathbb{Z}^2 : x^2 + y^2 \le r, (x,y)=1 \right\}.
	$$

	\begin{lemma}
		\label{Cn}
		Let $r > 0$ and $\epsilon > 0$ be real numbers. Then
		\begin{align*}
		|C(r)| &= \pi r + O(r^{\frac{1}{2}}); \\
		|C'(r)| &= \frac{6}{\pi} r + O\left(r^{\frac{1}{2}+\epsilon}\right);\\
		\sum_{(x,y)\in C'(r)}\!\!\!\sqrt{x^2 + y^2} &= \frac{4}{\pi} r\sqrt{r} + O\left(r^{1+\epsilon}\right).
		\end{align*}
	\end{lemma}
	
	\begin{proof}
		Since $C(\lfloor r \rfloor) \subseteq C(r) \subseteq C(\lceil r \rceil)$ and $C'(\lfloor r \rfloor) \subseteq C'(r) \subseteq C'(\lceil r \rceil)$, it suffices to prove these equalities for positive integers $r$.  
		
		Let $r > 0$ be an integer. The first two equalities are well-known in the literature as the Gauss circle problem and the primitive Gauss circle problem respectively; see \cite{hilbert1952geometry} and \cite{zhai1999number}. The third one follows from the first two because
		\begin{align*}
		\sum_{(x,y)\in C'(r)}\!\!\!\sqrt{x^2 + y^2}
		&= \sum_{i=1}^r \sqrt{i}\left(|C'(i)| - |C'(i-1)|\right)\\
		&= \sum_{i=1}^r \frac{6}{\pi}\sqrt{i} + \sum_{i=1}^{r-1}\left( \sqrt{i} - \sqrt{i+1} \right)\left( |C'(i)| - \frac{6}{\pi}i \right) + \sqrt{r}\left(|C'(r)|-\frac{6}{\pi}r\right)\\
		&= \frac{4}{\pi}r\sqrt{r} + O(r^{1+\epsilon}),
		\end{align*}
		where the last line follows from the fact that $\sqrt{i} - \sqrt{i+1} = O(\frac{1}{\sqrt{i}})$.
		\qed
	\end{proof}
		
\begin{proof}\textbf{of Theorem \ref{thm:order} (upper bound)}~	
    In view of \eqref{eq:incfreeisop}, in order to prove the upper bound in \eqref{eq:order} it suffices to prove
	\be
	\label{eq:nq}
	\bar{\alpha}(\Gamma_{n,q}) \ge \left(\frac{1}{2} - O(p^{\epsilon - \frac{1}{4}})\right)q^{\frac{n+1}{2}}
	\ee
	for any integer $n \ge 1$, prime power $q=p^e$ and real number $\epsilon > 0$. It turns out that the key step is to handle the special case when $n=2$ and $q$ is a prime.  
	
	\medskip
	\textsf{Case 1:} $n=2$ and $q=p$ is a prime. In this case \eqref{eq:nq} is equivalent to 
	\be
	\label{2p}
	\bar{\alpha}(\Gamma_{2,p}) \ge \frac{p\sqrt{p}}{2} - O\left(p^{\frac{5}{4}+\epsilon}\right).
	\ee 
We prove this by construction. Let
		$$
		S = \left\{\left<(x,y,1)\right>: (x,y) \in C\left(\frac{p\sqrt{p}}{2\pi}\right)\right\} 
		$$
		and
		$$
		T = \left\{\left<(a,b,c)\right>^\perp: a, b, c \in \mathbb{Z}, \left|c-\frac{p}{2}\right| < \frac{p}{2} - \frac{p^{\frac{3}{4}}}{\sqrt{2\pi}}\sqrt{a^2+b^2}\right\}.
		$$
		We first claim that $S \cup T$ is an independent set of $\Gamma_{2,p}$. Indeed, for any combination of $x,y,a,b,c$ as above, we have $0 < (x,y,1)\cdot(a,b,c)<p$, because
		\begin{align*}
		\left|ax+by+c-\frac{p}{2}\right| &\le |ax+by| + \left|c-\frac{p}{2}\right|\\
		&< \sqrt{x^2+y^2}\sqrt{a^2+b^2} + \frac{p}{2} - \frac{p^{\frac{3}{4}}}{\sqrt{2\pi}}\sqrt{a^2+b^2}\\
		&\le \frac{p}{2}.
		\end{align*}
Thus $S \cup T$ is an independent set of $\Gamma_{2,p}$.
		
		It follows directly from Lemma \ref{Cn} that $|S| = \frac{p\sqrt{p}}{2} + O\left(p^\frac{3}{4}\right)$. We can get a lower bound for $|T|$ by only picking the points where $(a,b)=1$ and identifying $(a,b,c)$ with $(-a,-b,p-c)$. Using this and Lemma \ref{Cn}, we obtain
		\begin{align*}
		|T|
		&\ge \frac{1}{2}\sum_{(x,y) \in C'\left(\frac{\pi}{2}\sqrt{p}\right)   }\left(p-1-2\frac{p^\frac{3}{4}}{\sqrt{2\pi}}\sqrt{x^2+y^2}\right)\\
		&= \frac{p-1}{2}\left|C'\left(\frac{\pi}{2}\sqrt{p}\right)\right|-\frac{p^\frac{3}{4}}{\sqrt{2\pi}}\sum_{(x,y) \in C'\left(\frac{\pi}{2}\sqrt{p}\right)}\sqrt{x^2+y^2}\\
		&= \frac{p-1}{2}\left(3\sqrt{p}+O\left(p^{\frac{1}{4}+\epsilon}\right)\right) - \frac{p^\frac{3}{4}}{\sqrt{2\pi}}\left(\sqrt{2\pi}p^\frac{3}{4}+O\left(p^{\frac{1}{2}+\epsilon}\right)\right)\\
		&= \frac{p\sqrt{p}}{2} - O\left(p^{\frac{5}{4}+\epsilon}\right).
		\end{align*}
		From this and the definition of $\bar{\alpha}$ we obtain \eqref{2p} immediately.	

		We now deal with the general case by using Lemma \ref{lem:pq} and what we proved in Case 1.
		
		\medskip
		\textsf{Case 2:} $n \ge 2$ is an integer and $q=p^e$ a prime power.
		
		By (e) in Lemma \ref{lem:pq}, $\incfree(\Gamma_{n,q}) \geq q^{\frac{n-1}{2}} \incfree(\Gamma_{1,q})$ for odd $n \ge 1$. This together with (a) in Lemma \ref{lem:pq} implies that \eqref{eq:nq} holds for any odd integer $n \ge 1$. Again, by (e) in Lemma \ref{lem:pq}, $\incfree(\Gamma_{n,q}) \geq q^{\frac{n}{2}-1} \incfree(\Gamma_{2,q})$ for even $n \ge 2$. Hence it suffices to prove \eqref{eq:nq} for $n=2$. 	
		This has been proved in \eqref{2p} when $e=1$. In general, if $e = 2k+1 \ge 3$ is odd, then by \eqref{2p} and (d) in Lemma \ref{lem:pq}, we have $\incfree(\Gamma_{2,q}) \geq p^{3k} \incfree(\Gamma_{2,p}) \ge \left(\frac{1}{2} - O(p^{\epsilon - \frac{1}{4}})\right) p^{\frac{3(2k+1)}{2}} = \left(\frac{1}{2} - O(p^{\epsilon - \frac{1}{4}})\right) q^{\frac{3}{2}}$ as required. If $e = 2k \ge 2$ is even, then by (c) in Lemma \ref{lem:pq}, $\incfree(\Gamma_{2,q}) \geq \frac{p^{3k}}{2} \ge \left(\frac{1}{2} - O(p^{\epsilon - \frac{1}{4}})\right) q^{\frac{3}{2}}$.  
		\qed
	\end{proof}

	\section{Proof of Theorem \ref{thm:16}}
	\label{sec:small}

	In this section we will first prove that the values of $\incfree(\Gamma_{2,q})$ in Table \ref{tab:small} are correct. When $q \leq 7$, the exact value of $\incfree(\Gamma_{2,q})$ is given in \cite{ure1996study}. We determine the values of $\incfree(\Gamma_{2,q})$ for $q=8, 9, 11, 13, 16$ by proving matching upper and lower bounds.
	\begin{lemma}
		\label{lem:ineq-x}
		Let $q$ be a prime power. Let $x$ be a positive integer and $m=\left\lfloor\frac{x}{q+1}\right\rfloor+1$. If
		$$
		\frac{x+1}{m+1}\left(2(q+1)-\frac{x}{m}\right) > q(q+1)-x,
		$$
		then
		$$\incfree(\Gamma_{2,q}) \le x.$$
	\end{lemma}
	\begin{proof}
		Let $S \subseteq V_1(\Gamma_{2,q})$ be such that $|S| = x+1$. Invoking Lemma \ref{lem:onesidedbound} with parameters $(v,k,\lambda)=(q^2+q+1,q+1,1)$ yields
		$|S| + |N(S)| > q^2+q+1$.
		Moreover, if $|S| > x+1$, then we can simply consider any subset of $S$ of size $x+1$ and the same result follows.
		\qed
	\end{proof}
	
	By Lemma \ref{lem:ineq-x}, we obtain $\incfree(\Gamma_{2,8}) \le 16$, $\incfree(\Gamma_{2,9}) \le 21$, $\incfree(\Gamma_{2,11}) \le 28$, $\incfree(\Gamma_{2,13}) \le 36$ and $\incfree(\Gamma_{2,16}) \le 52$ immediately. As we will see shortly, all these bounds except the second one are sharp. 
	
To obtain the sharp upper bound for $\incfree(\Gamma_{2,9})$, we use the classification \cite{coolsaet2012complete} of 3-arcs in $\PG(2,9)$. By its definition, a (17; 3)-arc is a subset $S$ of $V_1(\Gamma_{2,9})$ with $|S|=17$ such that any vertex in $V_2(\Gamma_{2,9})$ is adjacent to at most three vertices in $S$.
	
	\begin{lemma}
		\label{lem:72}
		Let $S$ be a (17; 3)-arc of $\PG(2,9)$. Then $|N(S)| \geq 72$.
	\end{lemma}
	\begin{proof}
		In \cite{coolsaet2012complete} it is shown that there are only four (17; 3)-arcs in $\PG(2,9)$ up to isomorphism. These can be given as coordinates on the affine plane, where $i$ denotes an element satisfying $i^2+1=0$:
		\begin{itemize}
			\item $S_1 = \{(0,0),(0,1),(0,2),(1,0),(1,1),(1,2),(2,i),(2,i+1),(2,2i),(i,0),(i,i+2),(i+1,1),(2i,2),(2i,i+1),(2i+2,i),(2i+2,i+1),(2i+2,2i)\}$
			\item $S_2 = \{(0,0),(0,1),(0,2),(1,0),(1,1),(1,2),(2,i),(2,i+1),(2,i+2),(i,0),(i,1),(i,2i+2),(i+1,2i+2),(i+2,2),(2i,i+1),(2i,2i),(2i,2i+2)\}$
			\item $S_3 = \{(0,0),(0,1),(0,2),(1,0),(1,1),(1,2),(2,i),(2,i+1),(2,2i+2),(i,0),(i,i),(i,2i+2),(i+1,i+2),(i+2,2),(i+2,i+1),(i+2,2i+1),(2i,1)\}$
			\item $S_4 = \{(0,0),(0,1),(0,2),(1,0),(1,1),(1,i),(2,0),(2,i),(2,i+1),(i,i),(i+1,2i+2),(2i,i+1),(2i+1,2),(2i+1,2i+2),(2i+2,2),(2i+2,i+1),(2i+2,2i+1)\}$
		\end{itemize}
		It is straightforward to check that $|N(S_i)|$ evaluates to 72, 73, 73 and 74 respectively.
		\qed
	\end{proof}
	
	\begin{lemma}
		$\incfree(\Gamma_{2,9}) \leq 19$.
	\end{lemma}
	\begin{proof}
		Suppose otherwise. Then there exists $S \in V_1(\Gamma_{2,q})$ such that $|S| = 20$ and $|N(S)| \leq 71$. Let $T_i$ denote the number of lines of $\PG(2, q)$ that are incident to exactly $i$ points in $S$. Using the same notation and technique as in the proof of Lemma \ref{lem:onesidedbound}, we obtain
		\begin{align*}
		\sum_{i=1}^{10}(i-2)(i-3)T_i
		&\leq \sum_{i=1}^{10}\left( i(i-1) - 4i + 6\right)T_i\\
		&\leq 20 \cdot (20-1)-4 \cdot 20 \cdot 10 +6 \cdot 71\\
		&=6,
		\end{align*}
		which implies that $(T_4 \leq 3 \mbox{ and } T_5 = 0)$ or $(T_4 = 0 \mbox{ and }  T_5 = 1)$ and $T_i = 0$ for all $i \geq 6$. Removing these (at most) three points from $S$ gives us a new (17; 3)-arc $S'$ with $|N(S)'| \leq 71$, contradicting Lemma \ref{lem:72}.
		\qed
	\end{proof}

To prove the values of $\incfree(\Gamma_{2, q})$ for $q \in \{8, 9, 11, 13, 16\}$ as shown in Table \ref{tab:small}, it suffices to show that $\incfree(\Gamma_{2,8}) \ge 16$, $\incfree(\Gamma_{2,9}) \ge 19$, $\incfree(\Gamma_{2,11}) \ge 28$, $\incfree(\Gamma_{2,13}) \ge 36$ and $\incfree(\Gamma_{2,16}) \ge 52$. When $q = 8$ we can use the construction in Lemma \ref{lem:pq}(e) to obtain $\incfree(\Gamma_{2,8}) \ge 16$. When $q = 16$ we can construct a maximal arc \cite{denniston1969some} to obtain $\incfree(\Gamma_{2,16}) \ge 52$. When $q \in \{9,11,13\}$, the following subset of $V_1(\Gamma_{2,q})$ yields $\incfree(\Gamma_{2,9}) \ge 19$, $\incfree(\Gamma_{2,11}) \ge 28$ and $\incfree(\Gamma_{2,13}) \ge 36$ respectively, where the elements are given as coordinates on the affine plane.
	
	$q=9$: $\{(0,0), (0,1), (0,i), (0,i+1), (1,0), (1,1), (1,i+1), (2,i), (2,2i+1), (i,i), (i,2i+2), (i+1,0), (i+1,i+2), (i+1,2i+2), (2i,i+2), (2i,2i+1), (2i,2i+2), (2i+2,1), (2i+2,2i+1)\}$, where $i$  denotes an element satisfying $i^2+1=0$.
	
	$q=11$: \{(0,0), (0,8), (0,10), (1,3), (1,7), (1,8), (2,5), (2,7), (2,10), (3,3), (3,5), (3,7),\\ (4,0), (4,8), (4,10), (5,1), (5,4), (5,9), (7,0), (7,4), (7,10), (9,1), (9,3), (9,9), (10,1), (10,4), (10,5), (10,9)\}.
	
	$q=13$: \{(0,0), (0,3), (0,6), (0,10), (1,1), (1,4), (1,10), (3,0), (3,5), (3,10), (3,11), (4,1), (4,2), (4,6), (4,7), (6,2), (6,3), (6,5), (6,12), (7,1), (7,3), (7,6), (7,12), (8,0), (8,6), (8,7), (8,11), (9,2), (9,4), (9,7), (10,4), (10,5), (10,12), (11,3), (11,4), (11,11)\}.

So far we have determined the values of $\incfree(\Gamma_{2,q})$ for $q \in \{8, 9, 11, 13, 16\}$. Combining these with the values of $\incfree(\Gamma_{2,q})$ for $q \in \{2,3,4,5,7\}$ given in \cite{ure1996study}, we obtain the results in the second column of Table \ref{tab:small}. In light of (\ref{eq:incfreeisop}), these give us upper bounds for $\isop(\Gamma_{2,q})$ as needed in the third column of Table \ref{tab:small}. The matching lower bound for $\isop(\Gamma_{2,q})$ seems to be difficult to obtain analytically, and so we run a program to achieve this. Since testing all subsets takes exponential time, we weaken some of the constraints and give a polynomial time program for the relaxed problem. 
	
	Define $h:[0,v]\rightarrow[0,2v]$ by
	$$h(x) = \frac{(q+1)^2 x}{q+x}+x$$
	and $g:\{0,1,\ldots,v\}\rightarrow\{0,1,\ldots,v\}$ by
	$$g(x) = \begin{cases}
	0 & x = 0\\
	\frac{x(2m(x)(q+1)-x+1)}{m(x)(m(x)+1)} & x \neq 0,
	\end{cases}$$
	where $m(x)=\left\lfloor\frac{q+x}{q+1}\right\rfloor$. 
	
	\begin{thm}
		\label{thm:prog}
		Let $q$ be a prime power and $v=q^2+q+1$. Then the optimal value of the following program is a lower bound for $\isop(\Gamma_{2,q})$. Furthermore, this problem can be solved in polynomial time with respect to $q$.
		\begin{align} 
                 \nonumber
		\text{minimize \quad }   & \frac{c+d}{a+b}  \\
		\text{subject to \quad } & a,b,c,d,e,f \in \{0,1,\ldots,v\} \label{1} \\ 
		& a+c+e=b+d+f=v \label{2} \\
		& \frac{h(a+b-h^{-1}(a+b))}{a+b} - 1 \le 1 - \frac{\incfree(\Gamma_{2,q})}{v} \label{3} \\
		& a \leq b \label{4} \\
		& b+d \ge g(a) \label{5} \\
		& a+c \ge g(b) \label{6} \\
		& c+e \ge g(f) \label{7} \\
		& \text{If } e\ge 1 \text{ then } qd \ge a \label{8} \\
		& \text{If } e=2 \text{ then } q(d+1) \ge 2(a+1) \label{9} \\
		& \text{If } q=5 \text{ and }a=9 \text{ then } b+d \ge 25 & \label{10} \\
		& \text{If } a>\incfree(\Gamma_{2,p}) \text{ then } b+d \ge v-\incfree(\Gamma_{2,p}) \label{11}
		\end{align}
 	\end{thm}
	
	\begin{proof}
		Let $S \subset V(\Gamma_{2,q})$ be such that $|S| \le v$ and $|N(S)|/|S| = \isop(\Gamma_{2,q})$. Let $A = S \cap V_1, B = S \cap V_2, C = N(S) \cap V_1, D = N(S) \cap V_2, E = V_1 \setminus (S \cup N(S))$ and $F = V_2 \setminus (S \cup N(S))$. Without loss of generality we may assume that $|A| \le |B|$. Let $a,b,c,d,e,f$ be the cardinalities of $A,B,C,D,E,F$ respectively. It suffices to check that all the eleven conditions are satisfied.
		
		Conditions \eqref{1} and \eqref{2} are trivially true. Condition \eqref{3} follows from (\ref{eq:incfreeisop}) and Theorem \ref{thm:f}. Condition \eqref{4} follows from our assumption that $|A| \le |B|$. Condition \eqref{5} follows from Lemma \ref{lem:ineq-x} and the fact that $N(A) \subseteq B \cup D$, and conditions \eqref{6} and \eqref{7} are similar.
		
		Note that if $e \ge 1$ then there is a common line in $D$ for each pair of points in $A \times E$. Each such line in $D$ then contains at most $q$ points in $A$. So $qd \geq a$ and condition \eqref{8} follows.
		
		Condition \eqref{9} uses a similar combinatorial argument, but we also take into account the fact that the two points in $E$ can have at most one line in $D$ joining them. By counting the number of 2-arcs from $A$ to $E$ (keeping in mind that all other lines in $D$ contain at most one point in $E$), we obtain the inequality $2(q-1)+q(d-1) \geq 2a$.
		
		Condition \eqref{10} is covered in \cite[Section 4.3.1]{ure1996study}. Finally, condition \eqref{11} follows from the definition of $\incfree$ and the fact that $N(A) \subseteq B \cup D$.
		
		We obtain the optimal value in polynomial time (with respect to $q$) by enumerating all $(q^2+q+1)^6$ combinations of $a,b,c,d,e,f$.
		\qed
	\end{proof}
	
	By running the program in Theorem \ref{thm:prog} for each prime power $q \le 16$ (see the appendix for the MAGMA code), we obtain a lower bound for $\isop(\Gamma_{2,q})$, which turns out to be exactly the same as the upper bound obtained from $\incfree(\Gamma_{2,q})$ via (\ref{eq:incfreeisop}). Therefore, for such $q$ the third column of Table \ref{tab:small} gives the exact values of $\isop(\Gamma_{2,q})$ and \eqref{eq:16} holds. This completes the proof of Theorem \ref{thm:16}.

	\appendix
	\section{MAGMA Code}
	\lstset{frame=tlrb,xleftmargin=\fboxsep,xrightmargin=-\fboxsep}
	The following code solves the program in Theorem \ref{thm:prog} by brute forcing through the entire sample space:
	\begin{lstlisting}
	for tup in [<2,2>, <3,3>, <4,6>, <5,7>, <7,13>,
	<8,16>, <9,19>, <11,28>, <13,36>, <16,52>] do
	q := tup[1];
	alph := tup[2];
	v := q^2 + q + 1;
	upperbound := 1 - alph/v;
	
	f := func<x | (q+1)^2 * x / (q + x) + x>;
	finv := func<x | (-b+Sqrt(b^2+4*q*x))/2
	where b is q^2+3*q+1-x>;
	sizes := {x : x in {1..v} |
	f(x-finv(x))/x - 1 le upperbound};
	
	g := func<x | x eq 0 select 0 else
	Ceiling(x*(2*M*(q+1)-x+1)/(M^2+M))
	where M is Floor((q + x) / (q + 1))>;
	
	"q =", q, ": i_v is in [",
	Min({(c+d)/(a+b) :
	c in {0..v-a},
	d in {0..v-b},
	a in {0..b} meet {x-b : x in sizes},
	b in {1..v} |
	(b + d ge g(a))
	and (a + c ge g(b))
	and (c + e ge g(f))
	and ((a le alph) or (b + d ge v - alph))
	and ((q ne 5) or (a ne 9) or (b + d ge 25))
	and ((e eq 0) or (d * q ge a))
	and ((e ne 2) or (d * q + q - 2 ge 2 * a))
	where e is v - a - c
	where f is v - b - d}),
	",", upperbound, "]";
	end for;
	\end{lstlisting}
 	
\medskip

\textbf{Acknowledgement}~~S. Zhou was supported by a Future Fellowship (FT110100629) of the Australian Research Council.  
	
	\bibliographystyle{plain}
	\bibliography{References.bbl}
\end{document}